\documentclass[a4paper,twoside,11pt,reqno]{amsart}

\usepackage[T1]{fontenc}
\usepackage[utf8]{inputenc}
\usepackage{lmodern}
\usepackage{amsmath,amssymb,mathtools}
\usepackage{microtype}
\usepackage{needspace}
\usepackage[colorlinks=true,linkcolor=blue,citecolor=blue,urlcolor=blue]{hyperref}
\hypersetup{
  pdftitle={Sharp hypercontractivity for free orthogonal quantum groups of Kac type},
  pdfauthor={Haonan Zhang},
  pdfsubject={Sharp hypercontractivity and Beckner inequalities for free orthogonal quantum groups of Kac type},
  pdfkeywords={free orthogonal quantum groups, Kac type, hypercontractivity, Beckner inequalities, logarithmic Sobolev inequalities, power majorants}
}

\numberwithin{equation}{section}

\theoremstyle{plain}
\newtheorem{theorem}{Theorem}[section]
\newtheorem{proposition}[theorem]{Proposition}
\newtheorem{lemma}[theorem]{Lemma}
\newtheorem{corollary}[theorem]{Corollary}

\theoremstyle{remark}
\newtheorem{remark}[theorem]{Remark}

\newcommand{\OF}{O_F^+}
\newcommand{\N}{\mathbb N}
\newcommand{\Nzero}{\mathbb N_0}
\newcommand{\C}{\mathbb C}
\newcommand{\1}{\mathbf 1}
\newcommand{\Pol}{\operatorname{Pol}}
\newcommand{\Ent}{\operatorname{Ent}}
\newcommand{\Tr}{\operatorname{Tr}}

\newcommand{\id}{\mathrm{id}}
\newcommand{\Fou}{\mathcal F}
\newcommand{\qint}[1]{[#1]}
\newcommand{\qfact}[1]{[#1]!}
\newcommand{\ip}[2]{\left\langle #1,#2\right\rangle}

\title[Sharp hypercontractivity for Kac-type $O_F^+$]
{Sharp hypercontractivity for free orthogonal quantum groups of Kac type}

\author[H. Zhang]{Haonan Zhang}
\address{Department of Mathematics, University of South Carolina, Columbia, SC 29208, USA}
\email{haonanzhangmath@gmail.com}


\begin{document}

\begin{abstract}
Let $N\geq2$, let $F\in GL_N(\mathbb C)$, and let $O_F^+$ be the associated free
orthogonal quantum group of Kac type, with tracial Haar state $h$. We prove that, for
$1<p\leq r<\infty$, the normalized heat semigroup $P_t=e^{-tL}$ on $O_F^+$ \cite{CFK14,BVY21} is hypercontractive with the optimal time
\[
 \|P_t:L_p(O_F^+)\to L_r(O_F^+)\|\leq1
 \quad\Longleftrightarrow\quad
 t\geq\frac12\log\frac{r-1}{p-1}.
\]
In particular, this solves a conjecture of Brannan, Vergnioux and Youn
\cite{BVY21}. To prove this hypercontractivity result, we establish the equivalent
logarithmic Sobolev inequality
\[
 \operatorname{Ent}(|x|^2)\leq2\mathcal E(x,x),
 \qquad x\in D(\mathcal E),
\]
with sharp constant $2$, where
$\operatorname{Ent}(y)=h(y\log y)-h(y)\log h(y)$ is the relative entropy and $\mathcal E(x,x)=\|L^{1/2}x\|_2^2$ is the Dirichlet form
associated with $P_t$. The main ingredients
are the cubic-majorant idea from the recent work of Frank--Ivanisvili
\cite{FI26} and Xie--Zhang \cite{XZC26,XZ26}, and a centered third-moment
estimate coming from representation theory.

We also prove the sharp Beckner inequalities
\[
 \frac{\|x\|_p^2-\|x\|_2^2}{p-2}\leq\mathcal E(x,x),
 \qquad 1\leq p\leq3,\quad p\ne2,
\]
whose limit as $p\to2$ recovers the logarithmic Sobolev inequality above.
Moreover, when $N=2$, the results can be sharpened by replacing the heat semigroup with
$\mathcal P_t=e^{-t\Lambda}$, where
$\Lambda=\sqrt{I+3L}-I$. The key step is to use the graded-twist realization
of $SU_{-1}(2)$ to transfer Beckner's sharp $p=3$ inequality on $SU(2)$
\cite{Beckner93}.
\end{abstract}

\maketitle
\enlargethispage{3pt}

\section{Introduction and main results}\label{sect:intro}

Hypercontractivity is a strong contraction property of Markov semigroups on
$L_p$-spaces. It provides, together with its equivalent logarithmic Sobolev formulation, a central tool
in probability, analysis and theoretical computer science \cite{Gro75,BE85,GZ03,ODonnell14}.
Carlen and Lieb's work on fermion hypercontractivity \cite{CL93}
has inspired a number of work in noncommutative harmonic analysis in this direction
\cite{Bia97,OZ99,JPPPR15,JPPP17,RX16,XZ26}. 
Here, the $L_p$-spaces are defined on von Neumann algebras, and the Markov semigroups usually
arise from noncommutative harmonic analysis and noncommutative probability. 

Compact quantum groups, introduced by Woronowicz \cite{Wor87,Wor98}, are
noncommutative analogues of compact groups. Among the fundamental
examples are the free orthogonal quantum groups $O_F^+$ of Van Daele and Wang
\cite{VDW96}, whose irreducible representations obey the fusion rules of
$SU(2)$ \cite{Ban96}. Here, $F$ is an $N\times N$ matrix satisfying certain conditions.
When $F=I_N$, $O_F^+$ is usually simply denoted by $O_N^+$. When studying the L\'evy processes on compact quantum groups,
Cipriani, Franz and Kula \cite{CFK14} introduced a family of central heat semigroups on $O_N^+$. These semigroups were subsequently extended to more general $O_F^+$ by Brannan, Vergnioux and Youn
\cite{BVY21}, and their hypercontractivity properties are the main focus of this paper. 

We recall the basic notation needed to
formulate our main results. Let $N\geq2$ and let $F\in GL_N(\C)$ satisfy
\begin{equation}\label{eq:F-hyp-intro}
  F\overline F=\varepsilon I_N,
  \qquad \varepsilon\in\{\pm1\},
  \qquad F^*F=I_N.
\end{equation}
The underlying $C^\ast$-algebra of the free orthogonal quantum group $\OF$ is the universal $C^\ast$-algebra generated by the coefficients of a unitary $u=(u_{ij})$ subject to
\[
  u=F\overline uF^{-1},
  \qquad \overline u=(u_{ij}^*).
\]
Under the assumption $F^\ast F=I_N$, $\OF$ is of Kac type, and its Haar
state $h$ is a trace, with respect to which we define the noncommutative
$L_p$-spaces $L_p(\OF)$ \cite{PX03}.

Let $u^k=(u_{ij}^k)_{1\le i,j\le d_k}$, $k\in\Nzero$, be the irreducible representations of $\OF$, with $u^0=\1$ and $u^1=u$. We write $\Pol(\OF)$ for the Hopf $*$-algebra of polynomial functions, which is the linear span of all $u^k_{ij}$ with $1\le i,j\le d_k$ and $k\in\Nzero$. It is dense in $L_p(\OF)$ for every $1\leq p<\infty$, and the normalized heat semigroup $P_t$ considered in this paper is determined by
\[
  P_t=e^{-tL},
  \qquad
  P_t(u_{ij}^k)=e^{-t\alpha_k}u_{ij}^k,
\]
where
\begin{equation}\label{eq:intro-alpha}
  \alpha_k=N\frac{U_k'(N)}{U_k(N)}.
\end{equation}
Here $(U_k)_{k\geq0}$ are the monic Chebyshev polynomials of the second kind,
\[
      U_0(x)=1,
  \quad U_1(x)=x,
  \quad U_{k+1}(x)=xU_k(x)-U_{k-1}(x).
\]
In particular, $\alpha_0=0$ and $\alpha_1=1$, and the spectral gap of $L$
is one.

\subsection{Sharp hypercontractivity and logarithmic Sobolev inequality}
\label{subsect:intro-hc-lsi}

\begin{theorem}[Sharp hypercontractivity]\label{thm:hc}
Let $N\geq2$, let $F\in GL_N(\C)$ satisfy \eqref{eq:F-hyp-intro}, and let
$1<p\leq r<\infty$. Then, for $t\geq0$, the normalized heat semigroup
$P_t=e^{-tL}$ satisfies
\begin{equation}\label{eq:hc-normalized}
  \|P_t:L_p(\OF)\to L_r(\OF)\|\leq1
  \quad\Longleftrightarrow\quad
  t\geq\frac12\log\frac{r-1}{p-1}.
\end{equation}
\end{theorem}

The necessity of
$t\geq\frac12\log\frac{r-1}{p-1}$ follows from a perturbation argument; the
challenge lies in the sufficient direction. For $r\geq2$, let
\[
 t_{N,r}^{\mathrm{opt}}
 :=\inf\bigl\{t\geq0:
 \|P_t:L_2(O_N^+)\to L_r(O_N^+)\|\leq1\bigr\}.
\]

For $F=I_N$, $\OF$ is the free orthogonal quantum group $O_N^+$. In this
case, some estimates of $t_{N,r}^{\mathrm{opt}}$ were obtained by Franz, Hong, Lemeux,
Ulrich and Zhang \cite{FHLUZ17} using ultracontractivity bounds and a
noncommutative martingale convexity inequality of Ricard and Xu \cite{RX16}.
Brannan, Vergnioux and Youn \cite{BVY21} subsequently improved these bounds
using noncommutative Khintchine inequalities derived from the rapid decay property and
Haagerup-type estimates, and they also treated general Kac-type $O_F^+$. These estimates
were not sharp.

The question of $t_{N,r}^{\mathrm{opt}}$ was asked in \cite{FHLUZ17}, and a more explicit 
conjecture was made in \cite{BVY21}. For $O_N^+$, Brannan, Vergnioux and Youn
\cite[Conjecture~4.8]{BVY21} conjectured, in the normalization
\eqref{eq:intro-alpha}, that for each $r>2$
\[
 t_{N,r}^{\mathrm{opt}}
 =\frac12\log(r-1)+\delta_{N,r},
 \qquad
 \delta_{N,r}=o(N^{-1})
 \quad(N\to\infty).
\]
Theorem~\ref{thm:hc} proves the stronger exact identity
$t_{N,r}^{\mathrm{opt}}=\frac12\log(r-1)$ for every $N\geq2$ and every
$r\geq2$. The same conclusion holds for every $\OF$ of Kac type, with the
optimal time defined in the same way.

If $N=2$ and $F\overline F=-I_2$, then
$O_F^+\cong SU(2)\cong S^3$ and $P_t=e^{(t/3)\Delta_{S^3}}$. In this case,
Theorem~\ref{thm:hc} becomes the sharp heat-semigroup
hypercontractivity theorem of Mueller and Weissler \cite{MW82}. In
particular, our main result provides a new proof of the $SU(2)$ case.

\medskip 

For the sufficient direction in Theorem~\ref{thm:hc}, it is enough to prove
the logarithmic Sobolev inequality. Let
\[
  \mathcal E(x,x):=\|L^{1/2}x\|_2^2,
  \qquad
  D(\mathcal E):=D(L^{1/2}).
\]
For $y\in L_1(\OF)_+$, set
\[
  \Ent(y):=h(y\log y)-h(y)\log h(y),
\]
where this relative entropy with respect to $h$ is understood as an extended
nonnegative quantity and $0\log0=0$.

\begin{theorem}[Sharp logarithmic Sobolev inequality]\label{thm:lsi}
Let $N\geq2$ and let $F$ satisfy \eqref{eq:F-hyp-intro}. For every
$x\in D(\mathcal E)$,
\begin{equation}\label{eq:lsi}
  \Ent(|x|^2)\leq2\mathcal E(x,x),
\end{equation}
and the constant $2$ is optimal.
\end{theorem}

The proof consists of two main ingredients. The first is the cubic-majorant
idea from the recent work of Frank--Ivanisvili on finite cycles \cite{FI26}
and of Xie--Zhang on cyclic groups \cite{XZC26} and free group von Neumann
algebras \cite{XZ26}. Its scalar input is the exact identity
\begin{equation}\label{eq:log-variational-intro}
 2s^2\log s
 =\min_{\lambda>0}
 \left\{
 \frac{2}{3\lambda}s^3+(2\log\lambda+1)s^2
 -2\lambda s+\frac{\lambda^2}{3}
 \right\},
 \qquad s>0.
\end{equation}
After applying the corresponding cubic majorant by functional calculus, the
logarithmic Sobolev inequality reduces to a centered third-moment bound,
which is the second main ingredient. This bound comes from a
representation-theoretic estimate. For $N\geq2$, let
\[
  V_k=\operatorname{span}\{u_{ij}^k\}\subset L_2(\OF),
\]
so that $L\vert_{V_k}=\alpha_k I_{V_k}$. For $u=u^*\in\Pol(\OF)$ with $h(u)=0$, define $S,Q\geq0$ by
\[
  S^2=h(u^2),
  \qquad
  Q^2=h\bigl(u(L-I)u\bigr).
\]
The desired third-moment estimate is
\begin{equation}\label{eq:third-moment-intro}
  |h(u^3)|\leq3S^2Q+Q^3.
\end{equation}
It is proved by encoding the relevant Clebsch--Gordan coefficients in a
matrix. The Kac assumption enters essentially through traciality, which
allows cyclic permutation of cubic products and a self-adjoint block
reduction.

For $r,s\geq1$, put
\[
  \eta_{rs}=\left(\frac{d_rd_s}{d_{r+s}}\right)^{1/2}.
\]
For $N\geq3$, we prove \eqref{eq:third-moment-intro} using the following
Clebsch--Gordan block estimate. If
$(i,j,k)=(s+t,r+t,r+s)$ is admissible, then
\begin{equation}\label{eq:intro-block}
  \|\Pi_k(ab)\|_2\leq c_{rst}\|a\|_2\|b\|_2,
  \qquad a\in V_i,\ b\in V_j,
\end{equation}
where $c_{rst}$ is the Clebsch--Gordan block coefficient. For $N\geq3$ these
coefficients satisfy the following estimates for $r,s,t\geq1$:
\[
  c_{rs0}=\eta_{rs},
  \qquad
  c_{rst}\leq\gamma_N^3\eta_{rs}\eta_{rt}\eta_{st},
  \qquad
  \gamma_N=\left(1+\frac1{N^2}\right)^{1/2}.
\]
The endpoint term $c_{rs0}=\eta_{rs}$ has no $\gamma_N$ loss. The interior
loss is absorbed by the eigenvalue estimate
\[
  \gamma_N^2\sum_{r=1}^{k-1}\eta_{r,k-r}^2\leq\alpha_k-1.
\]
Together these estimates give \eqref{eq:third-moment-intro}.

When $N=2$, the uniform interior factorization estimate for $c_{rst}$
degenerates. We instead cycle the factors and apply Hilbert--Schmidt
contractivity to the normalized Clebsch--Gordan isometry. For an admissible
triple $(i,j,k)=(s+t,r+t,r+s)$ and self-adjoint
$a\in V_i$, $b\in V_j$, and $c\in V_k$, this gives
\[
 |h(abc)|
 \leq\widetilde c_{rst}\|a\|_2\|b\|_2\|c\|_2,
 \qquad
 \widetilde c_{rst}
 =\min\left\{
 \sqrt{\frac{d_jd_k}{d_i}},
 \sqrt{\frac{d_id_k}{d_j}},
 \sqrt{\frac{d_id_j}{d_k}}
 \right\}.
\]
The symmetric coefficients satisfy the following estimates for
$r,s,t\geq1$:
\[
 \widetilde c_{rs0}=\eta_{rs},
 \qquad
 \widetilde c_{rst}
 \leq\gamma_2^3\eta_{rs}\eta_{rt}\eta_{st},
 \qquad
 \gamma_2=\frac{\sqrt5}{2}.
\]
The explicit eigenvalues $\alpha_k=k(k+2)/3$ absorb this loss, so the same
matrix argument proves \eqref{eq:third-moment-intro} for every Kac-type
$O_F^+$ with $N=2$. Finally, for a positive polynomial
$x=h(x)\1+u$, the logarithmic Sobolev argument closes with the choice
$\lambda=h(x)+Q$.

\subsection{Sharp Beckner inequalities}
\label{subsect:intro-beckner}

Similar ideas for power functions allow us to prove the following sharp
Beckner inequalities \cite{Beck89}.

\begin{theorem}[Sharp Beckner inequalities]\label{thm:beckner}
Let $N\geq2$ and let $F$ satisfy \eqref{eq:F-hyp-intro}. For every
$x\in D(\mathcal E)$ and every $p\in[1,3]\setminus\{2\}$,
\begin{equation}\label{eq:beckner}
  \frac{\|x\|_p^2-\|x\|_2^2}{p-2}
  \leq\mathcal E(x,x).
\end{equation}
The constant $1$ is optimal.
\end{theorem}

Theorem~\ref{thm:beckner} recovers Theorem~\ref{thm:lsi} by letting
$p\to2$. Indeed, the case $p=3$ implies that
$D(\mathcal E)\subset L_3(\OF)$, so differentiation at $p=2$ is justified
and gives
\[
 \lim_{p\to2}\frac{\|x\|_p^2-\|x\|_2^2}{p-2}
 =\frac12\Ent(|x|^2),
 \qquad x\in D(\mathcal E).
\]
Related sharp interpolation inequalities on spheres were proved in
\cite{Beckner93}.

The proof follows the same scheme as that of Theorem~\ref{thm:lsi}. If
$T_{p,\lambda}$ denotes the third-order Taylor polynomial of the function $s\mapsto s^p$ at
$\lambda>0$, then the signed Taylor remainder gives, for $s\geq0$,
\[
 s^p=\sup_{\lambda>0}T_{p,\lambda}(s)\quad(1<p<2),
 \qquad
 s^p=\inf_{\lambda>0}T_{p,\lambda}(s)\quad(2<p\leq3).
\]
These identities play the same role as the variational identity
\eqref{eq:log-variational-intro}. Applying the corresponding cubic minorant or majorant
by functional calculus reduces the proof to the same centered third-moment
estimate \eqref{eq:third-moment-intro}; the endpoint $p=1$ follows from the
spectral gap. Section~\ref{sect:beckner} carries out this argument.
Alternatively, the range $1\leq p<2$ follows from
Theorem~\ref{thm:hc} by Beckner's semigroup argument \cite{Beck89}.

\subsection{Sharper results for \texorpdfstring{$N=2$}{N=2}}
\label{subsect:intro-N2}

When $N=2$, we can use the specific structure to prove sharper results. 
Define the degree
generator and its Dirichlet form by
\[
 \begin{gathered}
  \Lambda|_{V_k}=kI_{V_k},\qquad \mathcal P_t=e^{-t\Lambda},\\
  \mathcal E_\Lambda(x,x)=\|\Lambda^{1/2}x\|_2^2,
  \qquad D(\mathcal E_\Lambda)=D(\Lambda^{1/2}).
 \end{gathered}
\]

The family $(\mathcal P_t)_{t\geq0}$ is a symmetric quantum Markov semigroup.
Indeed, by \eqref{eq:alpha-N2},
\[
 \Lambda=\phi(L),
 \qquad
 \phi(s)=\sqrt{1+3s}-1.
\]
The function $\phi$ is a Bernstein function. More explicitly, for $t>0$,
\begin{equation}\label{eq:N2-subordination}
 e^{-t\phi(s)}=\int_0^\infty e^{-3us}\,d\mu_t(u),
 \qquad
 d\mu_t(u)=e^t\frac{t}{2\sqrt\pi}u^{-3/2}
 e^{-t^2/(4u)-u}\,du,
\end{equation}
and the standard inverse-Gaussian integral shows that $\mu_t$ is a
probability measure. Consequently,
\[
 \mathcal P_t=e^{-t\Lambda}
 =\int_0^\infty P_{3u}\,d\mu_t(u).
\]
Thus $(\mathcal P_t)$ is normal, unital, completely positive, trace-preserving,
strongly continuous, and self-adjoint on $L_2(O_F^+)$, and hence is a
symmetric quantum Markov semigroup.

\begin{theorem}[Sharper inequalities for $N=2$]
\label{thm:N2-degree}
Let $N=2$ and let $F$ satisfy \eqref{eq:F-hyp-intro}. Then, for every
$x\in D(\mathcal E_\Lambda)$,
\begin{equation}\label{eq:N2-degree-lsi-intro}
 \Ent(|x|^2)\leq2\mathcal E_\Lambda(x,x),
\end{equation}
and, for every $p\in[1,3]\setminus\{2\}$,
\begin{equation}\label{eq:N2-degree-beckner-intro}
 \frac{\|x\|_p^2-\|x\|_2^2}{p-2}
 \leq\mathcal E_\Lambda(x,x).
\end{equation}
Moreover, for $1<p\leq r<\infty$ and $t\geq0$,
\begin{equation}\label{eq:N2-degree-hc-intro}
 \|\mathcal P_t:L_p(O_F^+)\to L_r(O_F^+)\|\leq1
 \quad\Longleftrightarrow\quad
 t\geq\frac12\log\frac{r-1}{p-1}.
\end{equation}
All constants are optimal.
\end{theorem}

Indeed, spectral calculus and \eqref{eq:alpha-N2} give
\begin{equation}\label{eq:N2-Lambda-vs-L-intro}
 \Lambda=\sqrt{I+3L}-I,
 \qquad
 (L-\Lambda)|_{V_k}=\frac{k(k-1)}3I_{V_k}.
\end{equation}
In particular, $\Lambda\leq L$, so
$D(\mathcal E)\subset D(\mathcal E_\Lambda)$. Thus, for
$x\in D(\mathcal E)$, the degree-form inequalities
\eqref{eq:N2-degree-lsi-intro} and
\eqref{eq:N2-degree-beckner-intro} imply the corresponding inequalities for
the heat semigroup, and their right-hand sides are strictly smaller whenever
$x$ has a nonzero $V_k$-component for some $k\geq2$. 

When $N=2$, there are essentially two quantum groups $\OF$: $SU(2)$ and $SU_{-1}(2)$. For the classical
branch $SU(2)\cong S^3$, the sharp logarithmic Sobolev inequality and
hypercontractivity for this Poisson semigroup are due to Beckner
\cite{Beckner92}, while the critical $p=3$ estimate used below is a case of
\cite[Theorem~6]{Beckner93}. Section~\ref{sect:N2-degree} transfers this
endpoint estimate to the graded twist $O_2^+\cong SU_{-1}(2)$ and derives
Theorem~\ref{thm:N2-degree} from the same centered third-moment criterion used
for the heat semigroup.

\subsection{Organization}
\label{subsect:intro-organization}

Section~\ref{sect2} recalls background and basic estimates for $O_F^+$.
Section~\ref{sect3} first proves the third-moment estimate for $N\geq3$, and
Section~\ref{sect:N2} then treats $N=2$. Section~\ref{sect:lsi} proves the
sharp logarithmic Sobolev inequality, and Section~\ref{sect:beckner} proves
the Beckner inequalities. Section~\ref{sect:N2-degree} proves the
stronger degree-semigroup result for $N=2$. The appendix records the passage
from positive polynomials to the full form domain.

\subsection*{Acknowledgments}
This work was partially supported by NSF grant DMS-2453408.
The author also acknowledges the use of ChatGPT~5.5 Pro.

\section{Preliminaries}\label{sect2}

\subsection{Compact quantum groups}

A compact quantum group is a pair $\mathbb G=(C(\mathbb G),\Delta)$, where
$C(\mathbb G)$ is a unital $C^*$-algebra and
\[
  \Delta:C(\mathbb G)\longrightarrow C(\mathbb G)\otimes_{\min}C(\mathbb G)
\]
is a coassociative unital $*$-homomorphism satisfying the cancellation
conditions
\[
  \overline{\Delta(C(\mathbb G))(1\otimes C(\mathbb G))}
  =C(\mathbb G)\otimes_{\min}C(\mathbb G)
  =\overline{\Delta(C(\mathbb G))(C(\mathbb G)\otimes1)}.
\]
It has a unique Haar state $h$, characterized by
\[
  (h\otimes\id)\Delta(a)=h(a)\1=(\id\otimes h)\Delta(a),
  \qquad a\in C(\mathbb G).
\]
Passing to the GNS representation of $h$ gives the reduced $C^\ast$-algebra
$C_r(\mathbb G)$ and the von Neumann algebra
$L_\infty(\mathbb G)=C_r(\mathbb G)''$, on which the Haar state $h$ is faithful. The compact quantum group is of Kac
type when $h$ is a trace, and we make this assumption throughout. For further
background on compact quantum groups, we refer to \cite{Wor87,Wor98,NT13}.

For $1\le p<\infty$, define
\[
        \|x\|_p:=h(|x|^p)^{1/p},
\]
and $L_p(\mathbb G)$ denotes the corresponding noncommutative $L_p$-space. We refer to \cite{PX03} for more background on noncommutative $L_p$-spaces.

A finite-dimensional unitary representation of $\mathbb G$ on a Hilbert space
$H_\alpha$ is a unitary
$u^\alpha=(u^\alpha_{ij})\in B(H_\alpha)\otimes C(\mathbb G)$ satisfying
$(\id\otimes\Delta)(u^\alpha)=u^\alpha_{12}u^\alpha_{13}$.  The coefficients
of irreducible representations span the Hopf $*$-algebra
$\Pol(\mathbb G)$, which is dense in $L_p(\mathbb G)$ for every
$1\leq p<\infty$.

\subsection{Free orthogonal quantum groups \texorpdfstring{$O_F^+$}{O(F)+}}

Let $N\ge2$ and let $F\in GL_N(\C)$ satisfy \eqref{eq:F-hyp-intro}. The compact quantum group $\OF$ is the free orthogonal quantum group of Van Daele and Wang \cite{VDW96}. The algebra $C(\OF)$ is the universal $C^*$-algebra generated by the entries of a unitary matrix $u=(u_{ij})\in M_N(C(\OF))$ satisfying
\begin{equation}\label{eq:OF-relation}
        u=F\overline uF^{-1},
        \qquad \overline u=(u_{ij}^*).
\end{equation}
The coproduct is determined by
\[
        \Delta(u_{ij})=\sum_{k=1}^N u_{ik}\otimes u_{kj}.
\]
We denote by $\Pol(\OF)\subset C(\OF)$ the Hopf $*$-algebra generated by the coefficients $u_{ij}$. Under the normalization $F\overline F=\pm I_N$, the compact quantum group $\OF$ is of Kac type exactly when $F$ is unitary \cite[Section~2]{BVY21}. Hence the Haar state $h$ is tracial.

The irreducible representations of $\OF$ are indexed by $k\in\Nzero$ \cite{Ban96}.  Let
\[
        u^k=(u_{ij}^k)\in B(H_k)\otimes C(\OF)
\]
be the irreducible representation of length $k$, with $H_0=\C$, $H_1=\C^N$, $u^0=\1$ and $u^1=u$. 
Since $\OF$ is of Kac type, the classical and quantum dimensions coincide;
we denote their common value by
\[
        d_k=\dim H_k.
\]
The fusion rules are
\begin{equation}\label{eq:fusion}
        u^i\otimes u^j
        \simeq
        u^{|i-j|}\oplus u^{|i-j|+2}\oplus\cdots\oplus u^{i+j}.
\end{equation}
In particular, the fusion is multiplicity-free.

The dimensions satisfy
\[
        d_0=1,\qquad d_1=N,\qquad d_{k+1}=Nd_k-d_{k-1}.
\]
Let $(U_k)_{k\geq0}$ be the monic Chebyshev polynomials of the second kind, characterized by $U_k(2\cos\theta)=\sin((k+1)\theta)/\sin\theta$, or equivalently by
\begin{equation}\label{eq:cheb}
        U_0(x)=1,\qquad U_1(x)=x,\qquad 
        U_{k+1}(x)=xU_k(x)-U_{k-1}(x).
\end{equation}
Then $d_k=U_k(N)$. For $N>2$, let $q\in(0,1)$ be determined by
\begin{equation}\label{eq:qparameter}
        N=q+q^{-1}.
\end{equation}
When $N=2$, set $q=1$ and interpret the quantum integers below by
continuity.  Thus
\[
        [m]_q=\frac{q^{-m}-q^m}{q^{-1}-q}
        \quad(q\neq1),
        \qquad
        [m]_q=m
        \quad(q=1).
\]
In the sequel we write $[m]$ for $[m]_q$ and $[m]!$ for
$[m][m-1]\cdots[1]$, with $[0]!=1$.  Then
\begin{equation}\label{eq:dimension-q}
        d_k=\qint{k+1}.
\end{equation}

Let
\[
        V_k=\operatorname{span}\{u_{ij}^k:1\le i,j\le d_k\}
    \subset L_2(\OF).
\]
We have the orthogonality relations 
\begin{equation}\label{eq:schur}
        h\bigl((u_{ij}^{k})^*u_{mn}^{\ell}\bigr)
        =\delta_{k\ell}\delta_{im}\delta_{jn}\,d_k^{-1},
\end{equation}
and the orthogonal decomposition 
\begin{equation}\label{eq:pw}
        L_2(\OF)=\bigoplus_{k\ge0}V_k.
\end{equation}
We denote by $\Pi_k$ the orthogonal projection onto $V_k$. Every $V_k$ is invariant under the involution, although the individual coefficients $u_{ij}^k$ need not be self-adjoint.

For $X\in B(H_k)$, define
\begin{equation}\label{eq:fourier-block}
        \Fou_k(X)=\sum_{i,j=1}^{d_k}X_{ji}u_{ij}^k.
\end{equation}
Then
\begin{equation}\label{eq:fourier-isometry}
        \|\Fou_k(X)\|_2=d_k^{-1/2}\|X\|_{S_2}.
\end{equation}
Suppose that $u^k$ occurs in $u^i\otimes u^j$. Since the fusion is
multiplicity-free, choose an isometric intertwiner
\[
        v_k^{i,j}:H_k\longrightarrow H_i\otimes H_j
\]
which is unique up to a scalar of modulus one. We shall use the following
product-block identity, obtained by decomposing $u^i\otimes u^j$ into its
multiplicity-free irreducible summands and compressing to the copy of $u^k$:

\begin{equation}\label{eq:product-block}
        \Pi_k\bigl(\Fou_i(X)\Fou_j(Y)\bigr)
        =\Fou_k\bigl((v_k^{i,j})^*(X\otimes Y)v_k^{i,j}\bigr).
\end{equation}
The right-hand side is independent of the phase chosen for $v_k^{i,j}$.

We call $(i,j,k)\in\Nzero^3$ admissible if
\[
        |i-j|\le k\le i+j,
        \qquad
        i+j+k\in2\Nzero.
\]
\Needspace{8\baselineskip}
For an admissible triple, set
\[
 r=\frac{j+k-i}{2},\qquad
 s=\frac{i+k-j}{2},\qquad
 t=\frac{i+j-k}{2}.
\]
Then $r,s,t\in\Nzero$ and
\begin{equation}\label{eq:param}
        (i,j,k)=(s+t,r+t,r+s).
\end{equation}
Thus $(i,j,k)$ are the three representation labels, while $(r,s,t)$ are
the corresponding Temperley--Lieb strand parameters.

\subsection{Clebsch--Gordan block coefficients}

For an admissible $(i,j,k)$, with $(r,s,t)$ as in \eqref{eq:param}, define
the theta function \cite[Section~3.4]{BCLY20} by
\begin{equation}\label{eq:theta}
 \Theta_{rst}
 =\frac{\qfact{r+s+t+1}\qfact r\qfact s\qfact t}
 {\qfact{r+s}\qfact{r+t}\qfact{s+t}}.
\end{equation}
The following lemma follows from
\cite[Equation~(3.9) in the proof of Proposition~3.4]{BVY21}, together with
the isometric normalization in \cite[Section~3.4]{BCLY20}. The same argument
applies at the endpoint $q=1$.

\begin{lemma}\label{lem:compression}
Let $(i,j,k)$ be admissible, with associated parameters $(r,s,t)$ as in
\eqref{eq:param}. Then, for $X\in B(H_i)$ and $Y\in B(H_j)$,
\begin{equation}\label{eq:compression}
 \|(v_k^{i,j})^*(X\otimes Y)v_k^{i,j}\|_{S_2(H_k)}
 \leq\frac{d_k}{\Theta_{rst}}
 \|X\|_{S_2(H_i)}\|Y\|_{S_2(H_j)}.
\end{equation}
\end{lemma}

Combining \eqref{eq:fourier-isometry}, \eqref{eq:product-block} and
Lemma~\ref{lem:compression} gives the Clebsch--Gordan estimate used
throughout the paper.

\begin{proposition}[Clebsch--Gordan block coefficient estimate]\label{prop:block}
Let $(i,j,k)=(s+t,r+t,r+s)$ be admissible. For $a\in V_i$ and $b\in V_j$,
\begin{equation}\label{eq:block}
  \|\Pi_k(ab)\|_2\leq c_{rst}\|a\|_2\|b\|_2,
\end{equation}
where the Clebsch--Gordan block coefficient is
\begin{equation}\label{eq:c-rst}
  c_{rst}=\frac{\sqrt{d_{r+s}d_{r+t}d_{s+t}}}{\Theta_{rst}}.
\end{equation}
The coefficient $c_{rst}$ is symmetric in $r,s,t$ and is independent of $F$ within the Kac class of fixed matrix size $N$.
\end{proposition}

\begin{proof}
Write $a=\Fou_i(X)$ and $b=\Fou_j(Y)$. By \eqref{eq:product-block} and
Lemma~\ref{lem:compression},
\begin{align*}
  \|\Pi_k(ab)\|_2
  &=d_k^{-1/2}\|(v_k^{i,j})^*(X\otimes Y)v_k^{i,j}\|_{S_2}\\
  &\leq d_k^{-1/2}\frac{d_k}{\Theta_{rst}}
  \|X\|_{S_2}\|Y\|_{S_2}\\
  &=\frac{\sqrt{d_id_jd_k}}{\Theta_{rst}}\|a\|_2\|b\|_2.
\end{align*}
This is \eqref{eq:block}. Symmetry follows from \eqref{eq:theta} and
\eqref{eq:c-rst}. Since $d_m=U_m(N)$ and $\Theta_{rst}$ depend only on $N$,
the same is true of $c_{rst}$, so it is independent of $F$ within the Kac
class of fixed matrix size $N$.
\end{proof}

\subsection{The normalized heat semigroup}
Recall that $\OF$ is of Kac type. The heat semigroups on $O_N^+$ considered here were introduced in \cite{CFK14}; see also \cite{FHLUZ17}. For the generalized heat semigroups on $O_F^+$, we refer to \cite{BVY21}. The generator used in \cite{BVY21} has eigenvalues $U_k'(N)/U_k(N)$. We rescale that generator by $N$ and write $P_t=e^{-tL}$; thus
\begin{equation}\label{eq:L-alpha}
        L\vert_{V_k}=\alpha_k I_{V_k},
        \qquad
        \alpha_k=N\frac{U_k'(N)}{U_k(N)}.
\end{equation}
Equivalently,
\[
        P_t(u_{ij}^k)=e^{-t\alpha_k}u_{ij}^k.
\]
The semigroup $P_t$ is normal, unital, completely positive, $h$-preserving and
symmetric on $L_2(\OF)$. In particular, 
\begin{equation}\label{eq:alpha01}
        \alpha_0=0,
        \qquad
        \alpha_1=1.
\end{equation}
\begin{lemma}\label{lem:alpha-lower}
For every $k\ge1$,
\begin{equation}\label{eq:alpha-lower}
        \alpha_k\ge k.
\end{equation}
\end{lemma}

\begin{proof}
Suppose first that $N>2$.  The roots of $U_k$ are
\[
        x_\ell=2\cos\frac{\ell\pi}{k+1}\le 2,
        \qquad 1\le\ell\le k,
\]
and $\sum_{\ell=1}^k x_\ell=0$. Hence
\begin{equation}\label{eq:root}
         \frac{U_k'(N)}{U_k(N)}
        =\sum_{\ell=1}^k\frac1{N-x_\ell}
        \ge \frac{k^2}{\sum_{\ell=1}^k(N-x_\ell)}
        =\frac{k}{N}  
\end{equation}
by Cauchy--Schwarz.

For $N=2$, the identity
\[
        U_k(2\cos\theta)=\frac{\sin((k+1)\theta)}{\sin\theta}
\]
can be differentiated at the endpoint through its Taylor expansion. As
$\theta\to0$,
\begin{align*}
 \frac{\sin((k+1)\theta)}{\sin\theta}
 &=(k+1)\left(1-\frac{k(k+2)}6\theta^2+O(\theta^4)\right),\\
 2\cos\theta&=2-\theta^2+O(\theta^4).
\end{align*}
On the other hand,
\[
 U_k(2\cos\theta)
 =U_k(2)+U_k'(2)(2\cos\theta-2)
   +O\bigl((2\cos\theta-2)^2\bigr).
\]
Comparing the constant and quadratic terms gives
\[
        U_k(2)=k+1,
        \qquad
        U_k'(2)=\frac{k(k+1)(k+2)}6.
\]
Consequently
\begin{equation}\label{eq:alpha-N2}
        \alpha_k=2\frac{U_k'(2)}{U_k(2)}
        =\frac{k(k+2)}3\geq k.
\end{equation}
This proves \eqref{eq:alpha-lower} for all $N\geq2$.
\end{proof}

Together with \eqref{eq:alpha01}, Lemma~\ref{lem:alpha-lower} shows that the spectral gap of $L$ is one.

\section{The third-moment estimate for \texorpdfstring{$N\geq3$}{N >= 3}}\label{sect3}

We first assume $N\geq3$ and prove the third-moment estimate
\eqref{eq:third-moment-intro}.
For $r,s\geq1$, define
\begin{equation}\label{eq:eta-gamma}
  \eta_{rs}=\left(\frac{d_rd_s}{d_{r+s}}\right)^{1/2},
  \qquad
  \gamma_N=\left(1+\frac1{N^2}\right)^{1/2}.
\end{equation}

\begin{lemma}\label{lem:factor}
For $r,s\geq1$,
\begin{equation}\label{eq:boundary-weight}
  c_{rs0}=\eta_{rs}.
\end{equation}
For $r,s,t\geq1$,
\begin{equation}\label{eq:interior-weight}
  c_{rst}\leq\gamma_N^2\eta_{rs}\eta_{rt}\eta_{st}
  \le \gamma_N^3\eta_{rs}\eta_{rt}\eta_{st}.
\end{equation}
\end{lemma}

\begin{proof}
If $t=0$, then \eqref{eq:theta} gives $\Theta_{rs0}=d_{r+s}$.  Hence
\[
  c_{rs0}
  =\frac{\sqrt{d_{r+s}d_rd_s}}{d_{r+s}}
  =\eta_{rs}.
\]

Now suppose $r,s,t\geq1$.  By \eqref{eq:theta}, \eqref{eq:c-rst} and \eqref{eq:eta-gamma},
\begin{equation}\label{eq:R-rst}
  R_{rst}:=\frac{c_{rst}}{\eta_{rs}\eta_{rt}\eta_{st}}
  =\frac{\qfact{r+s+1}\qfact{r+t+1}\qfact{s+t+1}}
  {\qfact{r+s+t+1}\qfact{r+1}\qfact{s+1}\qfact{t+1}}.
\end{equation}
For $a,b,c\in\Nzero$, the quantum-integer identity
\begin{equation}\label{eq:q-identity}
  \qint{a+b}\qint{a+c}-\qint{a}\qint{a+b+c}=\qint{b}\qint{c}
\end{equation}
follows directly from the definition of $\qint{\cdot}$. Applying it with $a=r+2$, $b=s$, and $c=t$, and using $s,t\geq1$, gives
\[
  \frac{R_{r+1,s,t}}{R_{rst}}
  =\frac{\qint{r+s+2}\qint{r+t+2}}{\qint{r+s+t+2}\qint{r+2}}
  =1+\frac{\qint{s}\qint{t}}
  {\qint{r+s+t+2}\qint{r+2}}
  \geq1.
\]
Thus $R_{rst}$ is nondecreasing in $r$. By symmetry, it is coordinatewise nondecreasing in $r,s,t$.

Put $z=q^2$ and $(z;z)_m=\prod_{\ell=1}^m(1-z^\ell)$.  Since
\[
  \qfact{m}=q^{-m(m-1)/2}\frac{(z;z)_m}{(1-z)^m},
\]
the powers of $q$ cancel in \eqref{eq:R-rst}, and
\begin{equation}\label{eq:R-poch}
  R_{rst}
  =(1-z)
  \frac{(z;z)_{r+s+1}(z;z)_{r+t+1}(z;z)_{s+t+1}}
  {(z;z)_{r+s+t+1}(z;z)_{r+1}(z;z)_{s+1}(z;z)_{t+1}}.
\end{equation}
By coordinatewise monotonicity,
\begin{equation}\label{eq:R-limit}
  R_{rst}\leq \lim_{m\to\infty}R_{mmm}
  =\prod_{\ell=2}^\infty(1-z^\ell)^{-1}.
\end{equation}
For $N\geq3$, one has
\[
  0<z=\left(\frac{2}{N+\sqrt{N^2-4}}\right)^2\leq\left(\frac{2}{3+\sqrt5}\right)^2<\frac16.
\]
The elementary product estimate gives
\begin{equation}\label{eq:product-bound}
  R_{rst}
  \leq \prod_{\ell=2}^\infty(1-z^\ell)^{-1}
  \leq \left(1-\sum_{\ell=2}^\infty z^\ell\right)^{-1}
  =\frac{1-z}{1-z-z^2}.
\end{equation}
Since $N^2=(1+z)^2/z$ and $z(1+z)^2\leq1-z-z^2$ for $0<z<1/6$,
\[
  \frac{1-z}{1-z-z^2}
  =1+\frac{z^2}{1-z-z^2}
  \leq1+\frac{z}{(1+z)^2}
  =1+\frac1{N^2}
  =\gamma_N^2.
\]
Therefore, $R_{rst}\leq \gamma_N^2$ and
\[
  c_{rst}\leq\gamma_N^2\eta_{rs}\eta_{rt}\eta_{st}
  \leq\gamma_N^3\eta_{rs}\eta_{rt}\eta_{st},
\]
which proves \eqref{eq:interior-weight}.
\end{proof}

The $\gamma_N$-loss in Lemma~\ref{lem:factor} is absorbed by the exact
eigenvalues:

\begin{lemma}[Eigenvalue absorption]\label{lem:eigenvalue-absorption}
For every $k\geq2$,
\begin{equation}\label{eq:eigenvalue-absorption}
  \gamma_N^2\sum_{r=1}^{k-1}\eta_{r,k-r}^2
  \leq \alpha_k-1.
\end{equation}
\end{lemma}

\begin{proof}
Recall that $d_k=U_k(N)$, and let us write $d_k':=U_k'(N)$.  Set
\[
  D_k=\sum_{r=1}^{k-1}d_rd_{k-r}.
\]
Since $\eta_{r,k-r}^2=d_rd_{k-r}/d_k$, the left-hand side of
\eqref{eq:eigenvalue-absorption} is $\gamma_N^2D_k/d_k$.  

Recall the generating function
\[
  \sum_{k\geq0}U_k(x)w^k=\frac1{1-xw+w^2}.
\]
Differentiating with respect to $x$ and then setting $x=N$ gives
\[
  \sum_{k\geq0}d_k'w^k
  =w\left(\sum_{k\geq0}d_kw^k\right)^2.
\]
Comparing the coefficient of $w^{k+1}$ yields
\begin{equation}\label{eq:derivative-convolution}
  d_{k+1}'=\sum_{i+j=k}d_id_j.
\end{equation}
The two endpoint terms on the right-hand side are $d_0d_k$ and $d_kd_0$.
Since $d_0=1$, this gives
\begin{equation}\label{eq:D-k-first}
  d_{k+1}'=2d_k+D_k.
\end{equation}
On the other hand, differentiating the recurrence
$U_{k+1}(x)=xU_k(x)-U_{k-1}(x)$ at $x=N$ gives
\begin{equation}\label{eq:differentiated-recurrence}
  d_{k+1}'=d_k+Nd_k'-d_{k-1}'.
\end{equation}
Combining \eqref{eq:D-k-first} and
\eqref{eq:differentiated-recurrence},
\begin{equation}\label{eq:D-k}
  D_k=Nd_k'-d_k-d_{k-1}'.
\end{equation}

Since $\eta_{r,k-r}^2=d_rd_{k-r}/d_k$ and
$\alpha_k=Nd_k'/d_k$, the desired inequality \eqref{eq:eigenvalue-absorption} is equivalent to
\begin{equation}\label{eq:eigenvalue-absorption-equivalent}
  \gamma_N^2D_k\leq Nd_k'-d_k.
\end{equation}
Using $\gamma_N^2=1+N^{-2}$ and \eqref{eq:D-k},
\begin{align*}
  N^2\bigl(Nd_k'-d_k-\gamma_N^2D_k\bigr)
  &=N^2\bigl(Nd_k'-d_k-D_k\bigr)-D_k=N^2d_{k-1}'-D_k.
\end{align*}
Now substitute \eqref{eq:D-k} once more and use the differentiated
recurrence \eqref{eq:differentiated-recurrence}
\[
        d_k'=d_{k-1}+Nd_{k-1}'-d_{k-2}'
\]
together with $d_k=Nd_{k-1}-d_{k-2}$.  This gives
\begin{align}
  N^2\bigl(Nd_k'-d_k-\gamma_N^2D_k\bigr)
  &=N^2d_{k-1}'-(Nd_k'-d_k-d_{k-1}')\notag\\
  &=(N^2+1)d_{k-1}'-Nd_k'+d_k\notag\\
  &=d_{k-1}'+Nd_{k-2}'-d_{k-2}\notag\\
  &=d_{k-1}'+(\alpha_{k-2}-1)d_{k-2}.
  \label{eq:eigenvalue-identity}
\end{align}
For $k=2$, the right-hand side is $d_1'-d_0=0$.  For $k\geq3$,
Lemma~\ref{lem:alpha-lower} gives $\alpha_{k-2}-1\geq0$.  Moreover, the root formula in \eqref{eq:root} gives
\[
  d_m'=d_m\sum_{\ell=1}^m\frac{1}{N-2\cos(\ell\pi/(m+1))}>0,
  \qquad m\geq1.
\]
Thus the right-hand side of \eqref{eq:eigenvalue-identity} is nonnegative,
so $Nd_k'-d_k-\gamma_N^2D_k\geq0$. This proves
\eqref{eq:eigenvalue-absorption-equivalent}, and hence
\eqref{eq:eigenvalue-absorption}.
\end{proof}

Now we are ready to prove the key third-moment estimate for $N\ge 3$.

\begin{proposition}\label{prop:third-ge3}
Let $u=u^*\in\Pol(\OF)$ satisfy $h(u)=0$. Define $S,Q\geq0$ by
\begin{equation}\label{eq:S-Q}
  S^2=h(u^2),
  \qquad
  Q^2=h\bigl(u(L-I)u\bigr).
\end{equation}
Then
\begin{equation}\label{eq:third-main}
  |h(u^3)|\leq3S^2Q+Q^3.
\end{equation}
\end{proposition}

\begin{proof}
Write
\[
  u=\sum_{k\geq1}u_k,
  \qquad
  u_k\in V_k,
  \qquad
  b_k=\|u_k\|_2.
\]
Since each $V_k$ is invariant under the involution and the decomposition is orthogonal, each $u_k$ is self-adjoint. Also, only finitely many components $u_k$ are nonzero, because $u\in\Pol(\OF)$; hence all sums and matrices below are finite.
Put $b_0=0$. Then by orthogonality, 
\begin{equation}\label{eq:S-Q-blocks}
  S^2=\sum_{k\geq1}b_k^2,
  \qquad
  Q^2=\sum_{k\geq1}(\alpha_k-1)b_k^2.
\end{equation}

By the fusion rules, $h(u_iu_ju_k)$ vanishes unless $(i,j,k)$ is admissible. For an admissible triple $(i,j,k)=(s+t,r+t,r+s)$, orthogonality, Cauchy--Schwarz, and Proposition~\ref{prop:block} give
\begin{align*}
  |h(u_iu_ju_k)|
  &=|h(\Pi_k(u_iu_j)u_k)|\\
  &\leq \|\Pi_k(u_iu_j)\|_2\|u_k\|_2\\
  &\leq c_{rst}b_ib_jb_k
   =c_{rst}b_{r+s}b_{r+t}b_{s+t},
\end{align*}
and thus 
\begin{equation}\label{eq:triple-sum}
  |h(u^3)|
  \leq\sum_{r,s,t\geq0}
  c_{rst}b_{r+s}b_{r+t}b_{s+t}.
\end{equation}

Let $\mathbf b=(b_r)_{r\geq1}\in\ell_2(\N)$ and define the symmetric matrix
$K=(K_{rs})$ with
\begin{equation}\label{eq:K}
  K_{rs}=\gamma_N\eta_{rs}b_{r+s},
  \qquad r,s\geq1.
\end{equation}
By Lemma~\ref{lem:eigenvalue-absorption},
\begin{align}
  \|K\|_{S_2}^2
  =\gamma_N^2\sum_{k\geq2}
    \left(\sum_{r=1}^{k-1}\eta_{r,k-r}^2\right)b_k^2\leq\sum_{k\geq2}(\alpha_k-1)b_k^2
  =Q^2.
  \label{eq:K-HS}
\end{align}

The terms in \eqref{eq:triple-sum} with exactly one of $r,s,t$ equal to zero contribute, by \eqref{eq:boundary-weight},
\begin{align}
  3\sum_{r,s\geq1}\eta_{rs}b_rb_sb_{r+s}
  =\frac3{\gamma_N}\ip{K\mathbf b}{\mathbf b} 
  \leq3\|K\|_{\mathrm{op}}\|\mathbf b\|_2^2
  \leq3QS^2.
  \label{eq:boundary-sum}
\end{align}
Terms with at least two zero coordinates vanish because $b_0=0$.

For the terms with $r,s,t\ge 1$, Lemma~\ref{lem:factor} yields
\begin{align}
  \sum_{r,s,t\geq1}c_{rst}b_{r+s}b_{r+t}b_{s+t}
  &\leq \sum_{r,s,t\geq1}\gamma_N^3\eta_{rs}\eta_{rt}\eta_{st}b_{r+s}b_{r+t}b_{s+t}\notag\\
  &=\sum_{r,s,t\geq1}K_{rs}K_{rt}K_{st}\notag\\
  &=\Tr(K^3)
  \leq\|K\|_{S_3}^3
  \leq\|K\|_{S_2}^3
  \leq Q^3.
  \label{eq:interior-sum}
\end{align}
Combining \eqref{eq:triple-sum}, \eqref{eq:boundary-sum} and \eqref{eq:interior-sum} proves \eqref{eq:third-main}.
\end{proof}

\section{The third-moment estimate for \texorpdfstring{$N=2$}{N = 2}}\label{sect:N2}

We now turn to $N=2$. Proposition~\ref{prop:block} remains valid, but the
uniform interior factorization of Lemma~\ref{lem:factor} degenerates. We use
instead a symmetric trilinear estimate obtained from the normalized
Clebsch--Gordan isometry.

\begin{lemma}[Trilinear block estimate for $N=2$]
\label{lem:N2-trilinear}
Assume $N=2$.  Let $(i,j,k)=(s+t,r+t,r+s)$ be admissible, and let
$a=a^*\in V_i$, $b=b^*\in V_j$, and $c=c^*\in V_k$.  Put
\begin{equation}\label{eq:N2-trilinear-coefficient}
 \widetilde c_{rst}
 =
 \min\left\{
 \sqrt{\frac{d_jd_k}{d_i}},
 \sqrt{\frac{d_id_k}{d_j}},
 \sqrt{\frac{d_id_j}{d_k}}
 \right\}.
\end{equation}
Then
\begin{equation}\label{eq:N2-trilinear}
 |h(abc)|
 \leq\widetilde c_{rst}\|a\|_2\|b\|_2\|c\|_2.
\end{equation}
With
\[
 \eta_{rs}=\left(\frac{d_rd_s}{d_{r+s}}\right)^{1/2},
 \qquad
 \gamma_2=\frac{\sqrt5}{2},
\]
one has
\begin{equation}\label{eq:N2-boundary-weight}
 \widetilde c_{rs0}=\eta_{rs},
 \qquad r,s\geq1,
\end{equation}
and
\begin{equation}\label{eq:N2-interior-weight}
 \widetilde c_{rst}
 \leq\gamma_2^3\eta_{rs}\eta_{rt}\eta_{st},
 \qquad r,s,t\geq1.
\end{equation}
\end{lemma}

\begin{proof}
By traciality, self-adjointness, and orthogonality,
\[
 |h(abc)|
 =|h(\Pi_k(ab)c)|
 \leq\|\Pi_k(ab)\|_2\|c\|_2.
\]
Write $a=\Fou_i(X)$ and $b=\Fou_j(Y)$.  If
$v_k^{i,j}:H_k\to H_i\otimes H_j$ is the normalized Clebsch--Gordan
isometry, then \eqref{eq:fourier-isometry} and
\eqref{eq:product-block} give
\begin{align*}
 \|\Pi_k(ab)\|_2
 &=d_k^{-1/2}
   \|(v_k^{i,j})^*(X\otimes Y)v_k^{i,j}\|_{S_2}\\
 &\leq d_k^{-1/2}\|X\otimes Y\|_{S_2}\\
 &=\sqrt{\frac{d_id_j}{d_k}}\,\|a\|_2\|b\|_2.
\end{align*}
Here compression by an isometry is contractive for the Hilbert--Schmidt norm.
Cyclically projecting onto $V_i$ or $V_j$ gives the other two
bounds in \eqref{eq:N2-trilinear-coefficient}, and hence
\eqref{eq:N2-trilinear}.

Now $d_m=m+1$ by \eqref{eq:dimension-q} with $q=1$. If $t=0$, then $r+s$ is the
largest index, so \eqref{eq:N2-trilinear-coefficient} gives
\[
 \widetilde c_{rs0}
 =\sqrt{\frac{d_rd_s}{d_{r+s}}}
 =\eta_{rs}.
\]
Suppose $r,s,t\geq1$.  By symmetry, assume that
$t=\min\{r,s,t\}$.  Then $r+s$ is the largest index, and division by
$\eta_{rs}\eta_{rt}\eta_{st}$ gives
\[
 \frac{\widetilde c_{rst}}
 {\eta_{rs}\eta_{rt}\eta_{st}}
 =
 \frac{(r+t+1)(s+t+1)}
 {(r+1)(s+1)(t+1)}
 \leq\frac{(2t+1)^2}{(t+1)^3}
 \leq\frac98
 \leq\gamma_2^3.
\]
The penultimate inequality follows from
\[
 9(t+1)^3-8(2t+1)^2
 =(t-1)(9t^2+4t-1)\geq0.
\]
This proves \eqref{eq:N2-interior-weight}.
\end{proof}

\medskip
\noindent\textit{Remark.}
Although \eqref{eq:N2-trilinear} is valid for every $N$, it closes the
present matrix argument only when $N=2$. For $r=s=t=m$, the required
comparison is
\[
 \widetilde c_{mmm}\leq(\alpha_{2m}-1)^{3/2}.
\]
For $N=2$, the two sides are of order $m^{1/2}$ and $m^3$, respectively;
for fixed $N>2$, they are of order $q^{-m}$ and $m^{3/2}$, respectively.
Thus Proposition~\ref{prop:block} is needed in the $N>2$ proof.

\Needspace{11\baselineskip}
\begin{proposition}[Centered third-moment estimate for $N=2$]
\label{prop:N2-centered-third}
Assume $N=2$.  If $u=u^*\in\Pol(O_F^+)$ and $h(u)=0$, define $S,Q\geq0$ by
\[
 S^2=h(u^2),
 \qquad
 Q^2=h\bigl(u(L-I)u\bigr).
\]
Then
\begin{equation}\label{eq:N2-third-heat}
 |h(u^3)|\leq3S^2Q+Q^3.
\end{equation}
\end{proposition}

\begin{proof}
Write
\[
 u=\sum_{k\geq1}u_k,
 \qquad
 u_k\in V_k,
 \qquad
 b_k=\|u_k\|_2,
\]
and put $b_0=0$.  Since each $V_k$ is invariant under the involution, every
$u_k$ is self-adjoint.  By \eqref{eq:alpha-N2},
\begin{equation}\label{eq:N2-SQ-blocks}
 S^2=\sum_{k\geq1}b_k^2,
 \qquad
 Q^2=\sum_{k\geq1}(\alpha_k-1)b_k^2,
 \qquad
 \alpha_k=\frac{k(k+2)}3.
\end{equation}
Let $\mathbf b=(b_r)_{r\geq1}$ and define the finitely supported symmetric matrix
\[
 K_{rs}=\gamma_2\eta_{rs}b_{r+s},
 \qquad r,s\geq1.
\]
For every $k\geq2$,
\[
 \sum_{r=1}^{k-1}\eta_{r,k-r}^2
 =\frac1{k+1}\sum_{r=1}^{k-1}(r+1)(k-r+1)
 =\frac{(k-1)(k+6)}6.
\]
Since $\gamma_2^2=5/4$,
\begin{equation}\label{eq:N2-eigenvalue-absorption}
 \alpha_k-1
 -\gamma_2^2\sum_{r=1}^{k-1}\eta_{r,k-r}^2
 =\frac{(k-1)(k-2)}8
 \geq0.
\end{equation}
Consequently,
\begin{align}
 \|K\|_{S_2}^2
 &=\gamma_2^2\sum_{r,s\geq1}\eta_{rs}^2b_{r+s}^2\notag\\
 &=\gamma_2^2\sum_{k\geq2}
   \left(\sum_{r=1}^{k-1}\eta_{r,k-r}^2\right)b_k^2\notag\\
 &\leq\sum_{k\geq2}(\alpha_k-1)b_k^2
 =Q^2.\label{eq:N2-K-HS}
\end{align}

The fusion rules and Lemma~\ref{lem:N2-trilinear} give
\[
 |h(u^3)|
 \leq
 \sum_{r,s,t\geq0}
 \widetilde c_{rst}b_{r+s}b_{r+t}b_{s+t}.
\]
The terms with exactly one of $r,s,t$ equal to zero are bounded by
\begin{align*}
 3\sum_{r,s\geq1}\eta_{rs}b_rb_sb_{r+s}
 =\frac3{\gamma_2}\langle K\mathbf b,\mathbf b\rangle
 \leq3\|K\|_{\mathrm{op}}\|\mathbf b\|_2^2
 \leq3QS^2.
\end{align*}
Terms with at least two zero coordinates vanish because $b_0=0$.  By
\eqref{eq:N2-interior-weight}, the terms with $r,s,t\geq1$ are bounded by
\begin{align*}
 \sum_{r,s,t\geq1}
 \gamma_2^3\eta_{rs}\eta_{rt}\eta_{st}
 b_{r+s}b_{r+t}b_{s+t}
 &=\operatorname{Tr}(K^3)\\
 &\leq\|K\|_{S_3}^3
 \leq\|K\|_{S_2}^3
 \leq Q^3.
\end{align*}
Combining the boundary and interior estimates proves
\eqref{eq:N2-third-heat}.
\end{proof}

The stronger degree-semigroup result and an alternative proof through the
graded-twist model are given in Section~\ref{sect:N2-degree}.

\section{The cubic majorant and the logarithmic Sobolev inequality}
\label{sect:lsi}

With the preceding estimates, the polynomial logarithmic Sobolev inequality
follows from the cubic-majorant strategy of \cite{XZ26}.

\begin{lemma}\label{lem:cubic}
For every $s\geq0$ and every $\lambda>0$,
\begin{equation}\label{eq:cubic-majorant}
  2s^2\log s
  \leq
  \frac{2}{3\lambda}s^3+(2\log\lambda+1)s^2
  -2\lambda s+\frac{\lambda^2}{3},
\end{equation}
where $2s^2\log s$ is interpreted as $0$ at $s=0$. Moreover,
\[
  2s^2\log s
  =\min_{\lambda>0}
  \left\{
  \frac{2}{3\lambda}s^3+(2\log\lambda+1)s^2
  -2\lambda s+\frac{\lambda^2}{3}
  \right\},
  \qquad s>0.
\]
\end{lemma}

\begin{proof}
For $s>0$, put $\tau=s/\lambda$. The right-hand side of
\eqref{eq:cubic-majorant} minus its left-hand side equals
$\lambda^2\phi(\tau)$, where
\[
  \phi(\tau)=\frac23\tau^3+\tau^2-2\tau+\frac13
  -2\tau^2\log\tau.
\]
\Needspace{6\baselineskip}
Differentiation gives
\[
  \phi'(\tau)=2\bigl(\tau^2-1-2\tau\log\tau\bigr).
\]
The expression in parentheses is increasing and vanishes at $\tau=1$.
Hence $\phi$ decreases on $(0,1)$ and increases on $(1,\infty)$. Thus
$\phi(\tau)\geq\phi(1)=0$. The case $s=0$ follows directly, and equality
for $s>0$ holds when $\lambda=s$.
\end{proof}

\begin{proof}[Proof of Theorem~\ref{thm:lsi}]
We first consider a positive polynomial $x$, where positivity is taken in
$L_\infty(\OF)$. By homogeneity, it suffices to assume $h(x^2)=1$. Write
\[
  x=a\1+u,
  \qquad a=h(x),
  \qquad h(u)=0.
\]
Then $u=u^*$ and $0<a\leq1$. Define $S,Q\geq0$ by
\[
  S^2=h(u^2)=1-a^2,
  \qquad
  Q^2=h\bigl(u(L-I)u\bigr).
\]
The centered third-moment estimate is Proposition~\ref{prop:N2-centered-third}
for $N=2$ and Proposition~\ref{prop:third-ge3} for $N\geq3$. Hence
\[
  h(u^3)\leq3S^2Q+Q^3.
\]
Set
\[
  E=h(xLx)=h(uLu)=S^2+Q^2,
  \qquad
  \lambda=a+Q.
\]
Then
\begin{align*}
  h(x^3)
  &=a^3+3aS^2+h(u^3)\\
  &\leq a^3+3a(1-a^2)+3(1-a^2)Q+Q^3\\
  &=3\lambda-3a\lambda^2+\lambda^3.
\end{align*}
Apply Lemma~\ref{lem:cubic} to $x$ by functional calculus. Since
$h(x^2)=1$,
\[
  \Ent(x^2)=h(2x^2\log x).
\]
Therefore
\begin{align*}
  \Ent(x^2)
  &\leq \frac{2}{3\lambda}h(x^3)
        +(2\log\lambda+1)-2a\lambda+\frac{\lambda^2}{3}\\
  &\leq \lambda^2-4a\lambda+3+2\log\lambda.
\end{align*}
On the other hand, using $E=1-a^2+Q^2$ and $\lambda=a+Q$,
\[
  2E-\bigl(\lambda^2-4a\lambda+3+2\log\lambda\bigr)
  =\lambda^2-1-2\log\lambda\geq0.
\]
Thus
\begin{equation}\label{eq:lsi-polynomial}
  \Ent(x^2)\leq2\mathcal E(x,x),
  \qquad x\in\Pol(\OF),\quad x\geq0.
\end{equation}
Proposition~\ref{prop:analytic-closure} in the appendix extends
\eqref{eq:lsi-polynomial} to every
$x\in D(\mathcal E)$.

Finally, choose $0\neq v=v^*\in V_1$ and set
$x_\varepsilon=\1+\varepsilon v$, which is positive for all sufficiently
small real $\varepsilon$. Since $h(v)=0$ and $Lv=v$,
\[
  \Ent(x_\varepsilon^2)
  =2\varepsilon^2h(v^2)+O(\varepsilon^3),
  \qquad
  \mathcal E(x_\varepsilon,x_\varepsilon)
  =\varepsilon^2h(v^2).
\]
Thus no constant smaller than $2$ can hold in \eqref{eq:lsi}.
\end{proof}

Theorem~\ref{thm:hc} now follows from Theorem~\ref{thm:lsi} by the standard
tracial Gross argument \cite{Gro75}; specifically,
\cite[Theorem~5.5]{OZ99} supplies the $L_p$-regularity required by
\cite[Theorem~3.8]{OZ99}. In our normalization this gives
$r(t)-1=(p-1)e^{2t}$. For necessity, choose $0\ne v=v^*\in V_1$ and set
$x_\varepsilon=\1+\varepsilon v$. For sufficiently small real $\varepsilon$,
\begin{align*}
 \|P_tx_\varepsilon\|_r^2
 &=1+(r-1)e^{-2t}\varepsilon^2\|v\|_2^2+O(\varepsilon^3),\\
 \|x_\varepsilon\|_p^2
 &=1+(p-1)\varepsilon^2\|v\|_2^2+O(\varepsilon^3).
\end{align*}
Thus contractivity forces $(r-1)e^{-2t}\leq p-1$, which is the lower bound
in \eqref{eq:hc-normalized}.

\section{A cubic majorant for power functions and Beckner inequalities}
\label{sect:beckner}

We prove the ranges $1<p<2$ and $2<p\leq3$ in
Theorem~\ref{thm:beckner} directly from the centered third-moment estimate;
the endpoint $p=1$ follows from the spectral gap. The cubic Taylor polynomial is a
majorant for $2<p\leq3$ and a minorant for $1<p<2$.

For $\lambda>0$, let $T_{p,\lambda}$ be the third-order Taylor polynomial of
$s\mapsto s^p$ at $s=\lambda$. In Taylor form,
\begin{align}
 T_{p,\lambda}(s)
 ={}&\lambda^p+p\lambda^{p-1}(s-\lambda)
 +\frac{p(p-1)}2\lambda^{p-2}(s-\lambda)^2\notag\\
 &+\frac{p(p-1)(p-2)}6\lambda^{p-3}(s-\lambda)^3.
 \label{eq:power-taylor-centered}
\end{align}
Equivalently, written in powers of $s$,
\begin{align}
 T_{p,\lambda}(s)
 ={}&\frac{p(p-1)(p-2)}6\lambda^{p-3}s^3
 +\frac{p(p-1)(3-p)}2\lambda^{p-2}s^2\notag\\
 &+\frac{p(p-2)(p-3)}2\lambda^{p-1}s
 +\frac{(p-1)(p-2)(3-p)}6\lambda^p.
 \label{eq:power-taylor}
\end{align}

\begin{lemma}[Cubic comparison for powers]\label{lem:power-majorant}
Let $s\geq0$ and $\lambda>0$. If $1<p<3$, $s>0$, and $s\ne\lambda$,
then there exists
$\xi_{s,\lambda}\in(\min\{s,\lambda\},\max\{s,\lambda\})$ such that
\begin{equation}\label{eq:power-remainder}
 s^p-T_{p,\lambda}(s)
 =\frac{p(p-1)(p-2)(p-3)}{24}
   \xi_{s,\lambda}^{p-4}(s-\lambda)^4.
\end{equation}
Consequently,
\[
 T_{p,\lambda}(s)\leq s^p\quad(1<p<2),
 \qquad
 T_{p,\lambda}(s)\geq s^p\quad(2<p<3).
\]
For $p\in\{1,2,3\}$, one has $T_{p,\lambda}(s)=s^p$.
\end{lemma}

\begin{proof}
Let $f_p(s)=s^p$. For $s>0$ and $s\ne\lambda$, Taylor's theorem with
Lagrange remainder gives
\[
 f_p(s)-T_{p,\lambda}(s)
 =\frac{f_p^{(4)}(\xi_{s,\lambda})}{4!}(s-\lambda)^4
\]
for some $\xi_{s,\lambda}$ strictly between $s$ and $\lambda$. Since
\[
 f_p^{(4)}(u)=p(p-1)(p-2)(p-3)u^{p-4},
 \qquad u>0,
\]
this is \eqref{eq:power-remainder}. Its sign is the sign of
$(p-2)(p-3)$, which proves the two comparisons for $s>0$; equality holds
when $s=\lambda$. The comparisons at $s=0$ follow by continuity. 
\end{proof}

\Needspace{16\baselineskip}
\begin{proposition}[Beckner inequalities from a centered third moment]
\label{prop:beckner-from-third}
Let $(\mathcal M,\tau)$ be a finite von Neumann algebra with a faithful normal
tracial state, and let $A$ be the nonnegative self-adjoint generator of a
symmetric Markov semigroup on $\mathcal M$, with $A\1=0$. Let
$\mathcal A\subset \mathcal M\cap D(A)$ be a unital $*$-subalgebra and put
\[
 \mathcal E_A(x,x)=\langle x,Ax\rangle_2.
\]
Assume that every $u=u^*\in\mathcal A$ with $\tau(u)=0$ satisfies
\[
 \mathcal E_A(u,u)\geq\tau(u^2)
\]
and, on defining $S,Q\geq0$ by
\[
 S^2=\tau(u^2),
 \qquad
 Q^2=\mathcal E_A(u,u)-S^2,
\]
one has
\begin{equation}\label{eq:abstract-third-moment}
 |\tau(u^3)|\leq3S^2Q+Q^3.
\end{equation}
Then, for every $x\in\mathcal A_+$ and every
$p\in[1,3]\setminus\{2\}$,
\begin{equation}\label{eq:abstract-beckner}
 \frac{\|x\|_p^2-\|x\|_2^2}{p-2}
 \leq\mathcal E_A(x,x).
\end{equation}
\end{proposition}

\begin{proof}
The assertion is trivial for $x=0$, so suppose that $x\ne0$. Put
\begin{equation}\label{eq:beckner-data}
 a=\tau(x),\qquad b=\tau(x^2),\qquad u=x-a\1,
 \qquad E=\mathcal E_A(x,x).
\end{equation}
Faithfulness of $\tau$ gives $a>0$. By the hypotheses,
\begin{equation}\label{eq:beckner-SQ}
 S^2=\tau(u^2)=b-a^2,
 \qquad Q^2=E-S^2\geq0.
\end{equation}
Choose
\begin{equation}\label{eq:beckner-lambda}
 \lambda=a+Q>0.
\end{equation}
The centered third-moment hypothesis gives
\begin{align}
 \tau(x^3)
 &=a^3+3aS^2+\tau(u^3)\notag\\
 &\leq a^3+3(a+Q)S^2+Q^3\notag\\
 &=3b\lambda-3a\lambda^2+\lambda^3.
 \label{eq:beckner-third}
\end{align}
The definitions also give the exact identity
\begin{equation}\label{eq:beckner-energy}
 E=b+\lambda^2-2a\lambda.
\end{equation}

For $p=1$, the spectral-gap assumption immediately gives
\[
 \|x\|_2^2-\|x\|_1^2=b-a^2=S^2\leq E,
\]
which is \eqref{eq:abstract-beckner}. Now fix
$p\in(1,3]\setminus\{2\}$ and set
\begin{equation}\label{eq:beckner-B}
 B=b+(p-2)E.
\end{equation}
\Needspace{18\baselineskip}
Define
\begin{align}
 G:={}&\frac{p(p-1)(p-2)}6\lambda^{p-3}
       \bigl(3b\lambda-3a\lambda^2+\lambda^3\bigr)+\frac{p(p-1)(3-p)}2b\lambda^{p-2}\notag\\
 &+\frac{p(p-2)(p-3)}2a\lambda^{p-1}+\frac{(p-1)(p-2)(3-p)}6\lambda^p\notag\\
 =&\frac{p(p-1)}2b\lambda^{p-2}
   -p(p-2)a\lambda^{p-1}
   +\frac{(p-1)(p-2)}2\lambda^p.
 \label{eq:beckner-G-definition}
\end{align}
The last equality, together with \eqref{eq:beckner-energy} and \eqref{eq:beckner-B}, gives
\begin{equation}
 G=\frac p2B\lambda^{p-2}-\frac{p-2}{2}\lambda^p.
 \label{eq:beckner-G}
\end{equation}
Combining Lemma~\ref{lem:power-majorant},
\eqref{eq:beckner-third} and \eqref{eq:beckner-G-definition}, we have
\begin{equation}\label{eq:beckner-directions}
 \tau(x^p)\begin{cases}
  \geq G,&1<p<2,\\
  \leq G,&2<p\leq3.
 \end{cases}
\end{equation}
Suppose first that $1<p<2$. If $B\leq0$, then
$\|x\|_p^2\geq B$ is automatic. If $B>0$, then Young's inequality gives
\[
 G=\frac p2\bigl(B\lambda^{p-2}\bigr)
   +\frac{2-p}{2}\lambda^p
 \geq
 \bigl(B\lambda^{p-2}\bigr)^{p/2}
 \bigl(\lambda^p\bigr)^{(2-p)/2}
 =B^{p/2}.
\]
Thus \eqref{eq:beckner-directions} gives
$\|x\|_p^2\geq B$, which is \eqref{eq:abstract-beckner} for $1<p<2$.

Now let $2<p\leq3$. Then $B>0$, and Young's inequality gives
\[
 B\lambda^{p-2}
 \leq\frac2pB^{p/2}+\frac{p-2}{p}\lambda^p.
\]
Consequently,
\[
 \tau(x^p)\leq G\leq B^{p/2},
\]
and hence $\|x\|_p^2\leq B$. This proves
\eqref{eq:abstract-beckner} for $2<p\leq3$.
\end{proof}

\begin{corollary}[Polynomial Beckner inequalities on $O_F^+$]
\label{cor:beckner-polynomial}
For every $x\in\Pol(\OF)$ that is positive in $L_\infty(\OF)$ and every
$p\in[1,3]\setminus\{2\}$,
\begin{equation}\label{eq:beckner-polynomial}
 \frac{\|x\|_p^2-\|x\|_2^2}{p-2}
 \leq\mathcal E(x,x).
\end{equation}
\end{corollary}

\begin{proof}
Apply Proposition~\ref{prop:beckner-from-third} with
\[
 (\mathcal M,\tau,A,\mathcal A)
 =\bigl(L_\infty(\OF),h,L,\Pol(\OF)\bigr).
\]
Lemma~\ref{lem:alpha-lower} gives the spectral-gap hypothesis.
Proposition~\ref{prop:N2-centered-third} gives
\eqref{eq:abstract-third-moment} when $N=2$, and
Proposition~\ref{prop:third-ge3} gives it when $N\geq3$.
\end{proof}

\begin{proof}[Proof of Theorem~\ref{thm:beckner}]
Corollary~\ref{cor:beckner-polynomial} proves the result for positive
polynomials throughout the stated range. Proposition~\ref{prop:analytic-closure}
in the appendix supplies the form-domain approximation and the
modulus reduction for arbitrary $x\in D(\mathcal E)$.

It remains to verify sharpness. This is a standard argument by choosing
$x_\varepsilon=\1+\varepsilon v$ for sufficiently
small real $\varepsilon$, where $0\ne v=v^*\in V_1$.
\end{proof}

We remark that the range $1\leq p<2$ can alternatively be derived from the
sharp hypercontractivity in Theorem~\ref{thm:hc} by Beckner's semigroup
argument \cite{Beck89}. Indeed, for $1<p<2$, set
$t=\frac12\log\frac1{p-1}$. Then
$\|P_tx\|_2\leq\|x\|_p$, while spectral calculus and the gap-one property give
\[
 \|x\|_2^2-\|P_tx\|_2^2
 =\langle x,(I-P_{2t})x\rangle_2
 \leq(1-e^{-2t})\mathcal E(x,x)
 =(2-p)\mathcal E(x,x).
\]
The case $p=1$ follows by continuity.

\begin{remark}[The classical $SU(2)$ case]\label{rem:beckner-su2}
If $N=2$ and $F\overline F=-I_2$, then
$O_F^+\cong SU(2)\cong S^3$, $L=-\Delta_{S^3}/3$, and
\eqref{eq:beckner} becomes
\[
 \frac{\|f\|_{L_p(S^3)}^2-\|f\|_{L_2(S^3)}^2}{p-2}
 \leq \frac13\int_{S^3}|\nabla f|^2\,d\mu,
 \qquad 1\leq p\leq3,\quad p\ne2.
\]
Classical spherical interpolation results give this inequality for the heat
semigroup by other methods and, on the upper side, extend to the critical
exponent $p=6$; see \cite{BE85,Beckner93,DEL17}. The present argument gives
the entire family of inequalities for the heat semigroup for
$1\leq p\leq3$ directly from the centered third-moment estimate and the
spectral gap. At $p=3$, this should not be confused with
Beckner's stronger
conformal multiplier inequality: the latter uses
$\Lambda|_{V_k}=kI$, whereas the heat form uses
$L|_{V_k}=k(k+2)I/3$. Section~\ref{sect:N2-degree} shows that the stronger
degree-form family persists for $N=2$.
\end{remark}

\par\medskip
\begin{remark}[Other examples]\label{rem:beckner-examples}
The same proof applies to the word-length Poisson semigroups on free group
von Neumann algebras \cite{XZ26}, the word-length generators on $\mathbb Z_n$,
$n\geq4$, and the Poisson generator on $\mathbb T$ \cite{XZC26}. The common
ingredients are the centered third-moment estimate,
Proposition~\ref{prop:beckner-from-third}, the standard closure argument, and
gap-one perturbations for sharpness. Fourier-multiplier domination supplies
the required third-moment estimate also for the normalized nearest-neighbor
generators on $\mathbb Z_n$, $n\geq4$, and the gap-one heat generator on
$\mathbb T$.
For the sharp logarithmic Sobolev inequality on nearest-neighbor cycles, see
\cite{FI26}.
\end{remark}

\section{Sharper results for \texorpdfstring{$N=2$}{N=2}}
\label{sect:N2-degree}

We prove Theorem~\ref{thm:N2-degree}. The graded-twist argument first
establishes the critical $p=3$ estimate for the degree generator $\Lambda$,
which is smaller than the heat generator $L$, and then recovers the
corresponding estimate for the heat semigroup by
\eqref{eq:N2-Lambda-vs-L-intro}.

We begin by recalling the two normal forms. This also makes clear that the
twisted case below exhausts the nonclassical possibilities at $N=2$.

\begin{lemma}[The two Kac normal forms of matrix size two]
\label{lem:N2-normal-forms}
Let $F\in U(2)$ satisfy $F\overline F=\varepsilon I_2$, where
$\varepsilon\in\{\pm1\}$. Then, as compact quantum groups,
\begin{equation}\label{eq:N2-classification}
 O_F^+\cong
 \begin{cases}
  O_2^+\cong SU_{-1}(2),&\varepsilon=1,\\
  SU(2),&\varepsilon=-1.
 \end{cases}.
\end{equation}
\end{lemma}

\begin{proof}
Unitarity and $F\overline F=\varepsilon I_2$ imply
$F^T=\varepsilon F$. Thus $F$ is symmetric when $\varepsilon=1$ and
skew-symmetric when $\varepsilon=-1$. In the first case, Takagi
factorization shows that $F$ is unitarily congruent to $I_2$. In the second,
$F$ is unitarily congruent, up to a scalar, to
\[
 J=\begin{pmatrix}0&1\\-1&0\end{pmatrix}.
\]
Unitary congruence and scalar multiplication of $F$ do not change the compact
quantum group, apart from the corresponding change of fundamental
coordinates. The identifications $O_2^+\cong SU_{-1}(2)$ and
$O_J^+\cong SU(2)$ are recorded in
\cite[Examples~3.3 and~4.10]{BNY16}.
\end{proof}

In the classical branch, identify $SU(2)$ with $S^3$ equipped with normalized
Haar measure $\nu$, and let $W_k$ denote the spherical harmonics of degree
$k$. Then
\[
 -\Delta_{S^3}|_{W_k}=k(k+2)I_{W_k},
 \qquad
 \Lambda=\sqrt{I-\Delta_{S^3}}-I.
\]
The $p=3$ case of Beckner's sharp spherical inequality
\cite[Theorem~6]{Beckner93} is
\begin{equation}\label{eq:N2-Beckner-endpoint}
 \|g\|_{L_3(S^3)}^2
 \leq\langle g,(\Lambda+I)g\rangle_2,
 \qquad g\in\Pol(SU(2)),\quad g=g^*.
\end{equation}
The full classical Poisson-semigroup hypercontractivity theorem was proved in
\cite{Beckner92}.

We next transfer \eqref{eq:N2-Beckner-endpoint} to the noncommutative normal
form. Let $\theta$ be the involutive automorphism of $SU(2)$ given by
\[
 \theta\begin{pmatrix}a&b\\c&d\end{pmatrix}
 =\begin{pmatrix}a&-b\\-c&d\end{pmatrix}.
\]
We use the same symbol for the induced pullback $*$-automorphism of
$\Pol(SU(2))$. It preserves Haar measure and every Peter--Weyl space $W_k$,
and hence commutes with $\Lambda$. The graded-twist realization gives
\begin{equation}\label{eq:N2-graded-twist}
 \Pol(O_2^+)=\bigoplus_{k\geq0}W_k\upsilon^k
 \subset\Pol(SU(2))\rtimes_\theta\mathbb Z_2,
 \qquad
 \upsilon^2=1,\qquad \upsilon f\upsilon=\theta(f).
\end{equation}
Here $\upsilon^k=1$ for even $k$ and $\upsilon^k=\upsilon$ for odd $k$.
In the regular crossed-product representation,
\begin{equation}\label{eq:N2-crossed-product-model}
 \rho(f+g\upsilon)
 =\begin{pmatrix}f&g\\\theta(g)&\theta(f)\end{pmatrix},
 \qquad
 h=(\nu\otimes\operatorname{tr}_2)\circ\rho,
\end{equation}
where $\operatorname{tr}_2$ is the normalized matrix trace. These facts
follow from \cite[Proposition~3.1 and Examples~3.3, 4.10]{BNY16}; see also
\cite[Theorem~1.1 and Proposition~4.1]{Bar15} for the concrete matrix model.

\begin{proposition}[Critical cubic estimate at $N=2$]
\label{prop:N2-degree-cubic}
For every Kac-type $O_F^+$ with $N=2$ and every
$y=y^*\in\Pol(O_F^+)$,
\begin{equation}\label{eq:N2-degree-cubic}
 |h(y^3)|
 \leq h\bigl(y(\Lambda+I)y\bigr)^{3/2}.
\end{equation}
\end{proposition}

\begin{proof}
For $O_F^+\cong SU(2)$, this follows immediately from
\eqref{eq:N2-Beckner-endpoint}, since
\[
 |h(y^3)|\leq h(|y|^3)=\|y\|_3^3
 \leq h\bigl(y(\Lambda+I)y\bigr)^{3/2}.
\]
It remains to consider
$O_F^+\cong O_2^+$. Write $y=a+b\upsilon$ in
\eqref{eq:N2-graded-twist}, where $a$ contains only even-degree components
and $b$ only odd-degree components. Since $y=y^*$,
\[
 a=a^*,
 \qquad
 \theta(b)=b^*.
\]
Let
\[
 \sigma_1=\begin{pmatrix}0&1\\1&0\end{pmatrix},\qquad
 \sigma_2=\begin{pmatrix}0&-i\\i&0\end{pmatrix},\qquad
 \sigma_3=\begin{pmatrix}1&0\\0&-1\end{pmatrix}
\]
be the Pauli matrices. They satisfy
\begin{equation}\label{eq:N2-Pauli-relations}
 \sigma_i\sigma_j+\sigma_j\sigma_i=2\delta_{ij}I_2,
 \qquad
 \operatorname{tr}_2(\sigma_i\sigma_j)=\delta_{ij},
 \qquad
 \operatorname{tr}_2(\sigma_i)=0.
\end{equation}
Put
\[
 f_0=\frac{a+\theta(a)}2,\qquad
 f_3=\frac{a-\theta(a)}2,\qquad
 f_1=\frac{b+\theta(b)}2,\qquad
 f_2=\frac{\theta(b)-b}{2i}.
\]
These functions are real-valued and
\begin{equation}\label{eq:N2-Pauli-decomposition}
 \rho(y)=f_0I_2+\sum_{j=1}^3f_j\sigma_j.
\end{equation}
The functions $f_0,f_3$ contain only even-degree components, whereas
$f_1,f_2$ contain only odd-degree components; moreover, $f_0,f_1$ are
$\theta$-even and $f_2,f_3$ are
$\theta$-odd. Since $\theta$ is unitary on $L_2(SU(2))$ and commutes with
$\Lambda$, the four functions are pairwise orthogonal for the quadratic form
of $\Lambda+I$.

For $\omega\in\{\pm1\}^3$, set
\[
 g_\omega=f_0+\sum_{j=1}^3\omega_jf_j,
 \qquad
 E_y=h\bigl(y(\Lambda+I)y\bigr).
\]
The matrix model intertwines $\Lambda$ with its entrywise amplification,
\[
 (\id_{M_2}\otimes\Lambda)\rho(y)=\rho(\Lambda y),
\]
and
\eqref{eq:N2-Pauli-relations} gives
\begin{equation}\label{eq:N2-equal-degree-energy}
 E_y=\sum_{j=0}^3\langle f_j,(\Lambda+I)f_j\rangle_2
 =\langle g_\omega,(\Lambda+I)g_\omega\rangle_2.
\end{equation}
The Pauli anticommutation relations also yield the pointwise identity
\begin{equation}\label{eq:N2-Rademacher-cubic}
 \operatorname{tr}_2(\rho(y)^3)
 =f_0^3+3f_0(f_1^2+f_2^2+f_3^2)
 =2^{-3}\sum_{\omega\in\{\pm1\}^3}g_\omega^3.
\end{equation}
Therefore, by \eqref{eq:N2-crossed-product-model},
\eqref{eq:N2-Beckner-endpoint}, and
\eqref{eq:N2-equal-degree-energy},
\[
 |h(y^3)|
 \leq2^{-3}\sum_\omega\int_{S^3}|g_\omega|^3\,d\nu
 \leq2^{-3}\sum_\omega E_y^{3/2}
 =E_y^{3/2}.
\]
\end{proof}

We now extract the centered estimate that connects the critical endpoint to
the power majorant. Let $u=u^*\in\Pol(O_F^+)$ satisfy $h(u)=0$, and set
\[
 S^2=h(u^2),
 \qquad
 Q_\Lambda^2=h\bigl(u(\Lambda-I)u\bigr)\geq0.
\]
For $R\geq0$, Proposition~\ref{prop:N2-degree-cubic} applied to
$R\1\pm u$ gives
\begin{equation}\label{eq:N2-degree-shift}
 \pm h(u^3)
 \leq(R^2+2S^2+Q_\Lambda^2)^{3/2}-R^3-3RS^2.
\end{equation}
If $Q_\Lambda>0$, choose $R=S^2/Q_\Lambda$. Then
$R^2+2S^2+Q_\Lambda^2=(R+Q_\Lambda)^2$, and
\eqref{eq:N2-degree-shift} becomes
\begin{equation}\label{eq:N2-degree-centered-third}
 |h(u^3)|\leq3S^2Q_\Lambda+Q_\Lambda^3.
\end{equation}
When $Q_\Lambda=0$, the same conclusion follows by letting
$R\to\infty$ in \eqref{eq:N2-degree-shift}.

\begin{proof}[Proof of Theorem~\ref{thm:N2-degree}]
The Markov property was proved through \eqref{eq:N2-subordination}.
The spectral gap of $\Lambda$ is one, and
\eqref{eq:N2-degree-centered-third} is precisely the centered third-moment
hypothesis in Proposition~\ref{prop:beckner-from-third}, with $A=\Lambda$.
That proposition gives \eqref{eq:N2-degree-beckner-intro} for positive
polynomials.

The proof of Proposition~\ref{prop:analytic-closure} applies with
\[
 (P_s,L,\alpha_k,\mathcal E)
 \quad\hbox{replaced by}\quad
 (\mathcal P_s,\Lambda,k,\mathcal E_\Lambda).
\]
Indeed, the $N=2$ estimate \eqref{eq:app-RD} and
\[
 \sum_{k\geq0}(k+1)^2e^{-2sk}<\infty
\]
give operator-norm smoothing, while dominated convergence gives convergence
in the form norm, since, for $x_k=\Pi_k(x)$,
\[
 \|\mathcal P_sx-x\|_2^2
 +\mathcal E_\Lambda(\mathcal P_sx-x,\mathcal P_sx-x)
 =\sum_{k\geq0}(1+k)(1-e^{-sk})^2\|x_k\|_2^2
 \longrightarrow0.
\]
Complete positivity of $\mathcal P_s$ supplies the amplified
modulus reduction. Thus \eqref{eq:N2-degree-beckner-intro} holds for every
$x\in D(\mathcal E_\Lambda)$. For positive
$x\in D(\mathcal E_\Lambda)$, the case $p=3$ ensures that $x\in L_3(O_F^+)$,
so letting $p\to2$ gives the logarithmic Sobolev inequality. The amplified
modulus argument then proves \eqref{eq:N2-degree-lsi-intro} for arbitrary
$x\in D(\mathcal E_\Lambda)$.

The tracial Gross argument \cite[Theorems~3.8 and~5.5]{OZ99} applied to
\eqref{eq:N2-degree-lsi-intro} gives the sufficient direction in
\eqref{eq:N2-degree-hc-intro}. For sharpness, again, choose
$0\ne v=v^*\in V_1$ and $x_\varepsilon=\1+\varepsilon v$ for sufficiently small real $\varepsilon$.
\end{proof}

\appendix
\stepcounter{section}
\section*{Appendix. Approximation and extension to the form domain}

We record the standard approximation and modulus arguments that pass the
polynomial logarithmic Sobolev and Beckner inequalities
to arbitrary elements of the form domain.

\begin{proposition}[Approximation and modulus reduction]
\label{prop:analytic-closure}
For every $x\in D(\mathcal E)$,
\[
 \Ent(|x|^2)\leq2\mathcal E(x,x).
\]
Moreover, for every $p\in[1,3]\setminus\{2\}$,
\[
 \frac{\|x\|_p^2-\|x\|_2^2}{p-2}
 \leq\mathcal E(x,x).
\]
\end{proposition}

\begin{proof}
We first construct positive polynomial approximants. The block estimate
\begin{equation}\label{eq:app-RD}
 \|a\|_\infty\leq C_N(k+1)\|a\|_2,
 \qquad a\in V_k,
\end{equation}
holds for a constant depending only on $N$. For $N\geq3$, this is
\cite[Proposition~3.4]{BVY21}, specialized to unitary $F$. For $N=2$, write
$a=\Fou_k(X)$. Since $u^k$ is unitary, \eqref{eq:fourier-isometry} and
$d_k=k+1$ give
\[
 \|a\|_\infty\leq\|X\|_{S_1}
 \leq d_k^{1/2}\|X\|_{S_2}
 =d_k\|a\|_2,
\]
so \eqref{eq:app-RD} holds with $C_2=1$.

Let $x\in D(\mathcal E)_+$, write $x_k=\Pi_k(x)$, and fix $s>0$. The
$L_2$-extension of $P_s$ preserves the positive cone.

Indeed, apply $P_s$ to bounded spectral truncations of $x$ and pass to the
$L_2$-limit. By
\eqref{eq:alpha-lower}, \eqref{eq:app-RD}, and Cauchy--Schwarz,
\begin{align*}
 \sum_{k\geq0}e^{-s\alpha_k}\|x_k\|_\infty
 &\leq C_N\sum_{k\geq0}(k+1)e^{-sk}\|x_k\|_2\\
 &\leq C_N
 \left(\sum_{k\geq0}(k+1)^2e^{-2sk}\right)^{1/2}\|x\|_2<\infty.
\end{align*}
Thus $\sum_{k\geq0}e^{-s\alpha_k}x_k$ converges in operator norm to
$P_sx\in L_\infty(\OF)_+$. Set
\[
 y_n=\sum_{k=0}^n e^{-s\alpha_k}x_k,
 \qquad
 \delta_n=\|y_n-P_sx\|_\infty,
 \qquad
 z_n=y_n+\delta_n\1.
\]
Every $y_n$ is self-adjoint, and
$y_n\geq P_sx-\delta_n\1\geq-\delta_n\1$; hence
$z_n\in\Pol(\OF)_+$. Moreover, $z_n\to P_sx$ in operator norm and, since
$L\1=0$,
\[
 \mathcal E(z_n-P_sx,z_n-P_sx)
 =\sum_{k>n}\alpha_ke^{-2s\alpha_k}\|x_k\|_2^2
 \longrightarrow0.
\]
It follows that the $L_p$ norms converge for every finite $p$. The spectra of
$z_n$ and $P_sx$ lie in a common compact interval for all large $n$, so
continuous functional calculus also gives
\[
 \Ent(z_n^2)\longrightarrow\Ent((P_sx)^2).
\]
Apply \eqref{eq:lsi-polynomial} and
Corollary~\ref{cor:beckner-polynomial} to $z_n$ and pass to the limit. For
every $p\in[1,3]\setminus\{2\}$,
\begin{equation}\label{eq:app-beckner-smoothed}
 \frac{\|P_sx\|_p^2-\|P_sx\|_2^2}{p-2}
 \leq\mathcal E(P_sx,P_sx),
\end{equation}
and
\begin{equation}\label{eq:app-lsi-smoothed}
 \Ent((P_sx)^2)\leq2\mathcal E(P_sx,P_sx).
\end{equation}

It remains to let $s\downarrow0$. Since $x\in D(\mathcal E)$,
\begin{align*}
 &\|P_sx-x\|_2^2+\mathcal E(P_sx-x,P_sx-x)\\
 &\qquad
 =\sum_{k\geq0}(1+\alpha_k)(1-e^{-s\alpha_k})^2\|x_k\|_2^2
 \longrightarrow0
\end{align*}
by dominated convergence. For $p>2$, $L_2$ convergence implies convergence
in measure, and lower semicontinuity of the $L_p$ norm gives
\[
 \|x\|_p\leq\liminf_{s\downarrow0}\|P_sx\|_p.
\]
Using the equivalent form
\[
 \|P_sx\|_p^2
 \leq\|P_sx\|_2^2+(p-2)\mathcal E(P_sx,P_sx)
\]
then proves the positive form-domain inequality for $2<p\leq3$. For
$1\leq p<2$, the tracial normalization gives
\[
 \|P_sx-x\|_p\leq\|P_sx-x\|_2\longrightarrow0.
\]
We may therefore pass directly to the limit in the equivalent form
\[
 \|P_sx\|_p^2
 \geq\|P_sx\|_2^2+(p-2)\mathcal E(P_sx,P_sx),
\]
which proves the positive form-domain inequality for $1\leq p<2$.

For the logarithmic endpoint,
\[
 \|(P_sx)^2-x^2\|_1
 \leq(\|P_sx\|_2+\|x\|_2)\|P_sx-x\|_2
 \longrightarrow0.
\]
Entropy is lower semicontinuous on $L_1(\OF)_+$. Indeed,
\[
 h(y\log y)=\sup_{b=b^*\in L_\infty(\OF)}
 \bigl\{h(yb)-h(e^{b-1})\bigr\},
 \qquad y\in L_1(\OF)_+.
\]
The upper bound in this formula follows from Klein's inequality, while
equality follows by bounded spectral truncations of $\1+\log y$. Hence
$y\mapsto h(y\log y)$ is lower semicontinuous, while
$h(y)\log h(y)$ is continuous in the $L_1$-norm. Consequently,
\[
 \Ent(x^2)
 \leq\liminf_{s\downarrow0}\Ent((P_sx)^2)
 \leq2\lim_{s\downarrow0}\mathcal E(P_sx,P_sx)
 =2\mathcal E(x,x).
\]
Thus both inequalities hold for every positive $x\in D(\mathcal E)$.

We next pass to arbitrary operators. First suppose that $Y=Y^*$ and write
$Y=Y_+-Y_-$, where $Y_+Y_-=0$. Positivity, symmetry, and traciality of
$P_t$ give
\begin{align*}
 \langle |Y|,P_t|Y|\rangle_2-
 \langle Y,P_tY\rangle_2
 &=2h(Y_+P_tY_-)+2h(Y_-P_tY_+)\\
 &=4h(Y_+P_tY_-)\geq0.
\end{align*}
Since $\||Y|\|_2=\|Y\|_2$, the generator-form limit yields
$\mathcal E(|Y|,|Y|)\leq\mathcal E(Y,Y)$.

Apply this contraction to the amplification
$\id_{M_2}\otimes P_t$ and to
\[
 Y=\begin{pmatrix}0&x\\x^*&0\end{pmatrix},
 \qquad
 |Y|=\begin{pmatrix}|x^*|&0\\0&|x|\end{pmatrix}.
\]
With the normalized matrix trace, this gives
\begin{equation}\label{eq:app-modulus}
 \frac12\bigl(
   \mathcal E(|x|,|x|)+\mathcal E(|x^*|,|x^*|)
 \bigr)
 \leq\mathcal E(x,x).
\end{equation}
Here we used $\mathcal E(x^*,x^*)=\mathcal E(x,x)$, which follows from
$*$-preservation of the semigroup. In particular, both moduli belong to
$D(\mathcal E)$. Apply the positive
inequalities to $|x|$ and $|x^*|$. Traciality and polar decomposition give
\[
 \||x|\|_p=\||x^*|\|_p=\|x\|_p
\]
and equality of the corresponding entropies. Averaging the two inequalities
and using \eqref{eq:app-modulus} proves the proposition.
\end{proof}

\end{document}